\documentclass[12pt]{amsart}
\usepackage{latexsym}
\usepackage{amsfonts,amssymb,amsmath,amsthm,amscd}
\usepackage[mathscr]{eucal}
\usepackage{graphicx}
\usepackage{xcolor}

\newcommand{\C}{\mbox{${\mathbb C}$}}

\newcommand{\I}{\mbox{${\mathbb I}$}}

\newcommand{\PP}{\mbox{${\mathbb P}$}}

\newcommand{\R}{\mbox{${\mathbb R}$}}
\newcommand{\Z}{\mbox{${\mathbb Z}$}}

\newcommand{\tr}{{\rm tr}}

\newcommand{\Ric}{{\rm Ric}}

\newcommand{\SU}{{\rm SU}}
\newcommand{\Sp}{{\rm Sp}}

\newcommand{\U}{{\rm U}}

\makeatletter
\def\numberwithin#1#2{\@ifundefined{c@#1}{\@nocnterrr}{%
  \@ifundefined{c@#2}{\@nocnterr}{%
  \@addtoreset{#1}{#2}%
  \toks@\expandafter\expandafter\expandafter{\csname the#1\endcsname}%
  \expandafter\xdef\csname the#1\endcsname
    {\expandafter\noexpand\csname the#2\endcsname
     .\the\toks@}}}}
\makeatother
\numberwithin{equation}{section}

\newtheorem{thm}[equation]{Theorem}
\newtheorem{lemma}[equation]{Lemma}
\newtheorem{prop}[equation]{Proposition}

\newtheorem{cor}[equation]{Corollary}

\newtheorem{ex}[equation]{Example}

\newtheorem{rem}[equation]{Remark}

\begin{document}

\title{On the Linear Stability of Nearly-K\"ahler $6$-manifolds}
\author{Changliang Wang}
\address{Max-Planck-Institut f\"ur Mathematik, Vivatsgasse 7, Bonn 53111, Germany}
\email{cwangmath@mpim-bonn.mpg.de}

\author{M. Y.-K. Wang}
\address{Department of Mathematics and Statistics, McMaster
University, Hamilton, Ontario, L8S 4K1, CANADA}
\email{wang@mcmaster.ca}

\date{revised \today}

\begin{abstract}
We show that a strict, nearly K\"ahler $6$-manifold with either second or third Betti number
nonzero is linearly unstable with respect to the $\nu$-entropy of Perelman and hence is dynamically
unstable for the Ricci flow.
\end{abstract}

\maketitle



\medskip
\setcounter{section}{0}

\section{\bf Introduction}

Manifolds which admit a non-trivial Killing spinor form a distinguished subclass of Einstein manifolds.
Recall that the Killing spinor equation is given by
\begin{equation} \label{KS-eqn}
 \nabla_X \sigma = c X \cdot \sigma
\end{equation}
where $\sigma$ is a complex spinor field, $c$ is a constant, $X$ is an arbitrary tangent vector, and
$\cdot$ denotes Clifford multiplication. Let $(M, g)$ be the underlying Riemannian spin manifold and $n$
be its (real) dimension. Since a Killing spinor is an eigenspinor for the Dirac operator:  $D \sigma = -nc \sigma$,
the constant $c$ is zero (parallel spinor case), purely imaginary, or real.

In the $c=0$ case, we obtain special geometries of Calabi-Yau, hyperk\"ahler, ${\rm G}_2$, and
${\rm Spin}(7)$ types. By the work of X. Dai, X. D. Wang, and G. Wei \cite{DWW05}, the underlying Ricci-flat metric
$g$ is linearly stable. When $c$ is purely imaginary, the manifolds were classified by H. Baum \cite{Bau89}.
By the work of Kr\"oncke \cite{Kr17} and the first author \cite{Wan17}, the Einstein metrics (with negative
scalar curvature) are also linearly stable.

When $c$ is real and nonzero, the Einstein metric $g$ has positive scalar curvature, and so by Lichnerowicz's
theorem it cannot admit any harmonic spinors. T. Friedrich \cite{Fr80} then derived a positive lower bound for the
eigenvalues of the square of the Dirac operator, and furthermore showed that the lower bound is
achieved precisely for those manifolds which admit a non-trivial Killing spinor. These manifolds are
known to be locally irreducible, and cannot be locally-symmetric unless they are spherical space-forms (which we will exclude from our discussion henceforth).
While they are far from being classified, there is a well-known rough classification by C. B\"ar \cite{Ba93}
in terms of the restricted holonomy of their metric cones $(\R_{+} \times M, dt^2 + t^2 g).$  The only
possibilities are $\SU(\frac{n+1}{2}), \Sp(\frac{n+1}{4}), {\rm G}_2$, or ${\rm Spin}(7)$. Thus $n$ can be even
only if $n=6$, and, in this case, by the work of Grunewald \cite{Gru90} (see also chapter 5 in \cite{BFGK91}),
$(M, g)$ is either isometric to round $S^6$ or a strict nearly K\"ahler $6$-manifold.

This article examines the linear stability of this class of Einstein $6$-manifolds.
Recall that a nearly K\"ahler manifold $(M, J, g)$ is an almost Hermitian manifold that satisfies
\begin{equation}  \label{NKcond}
  (\nabla_X J)X = 0
\end{equation}
for all tangent vectors $X$, where $\nabla$ denotes the Levi-Civita connection of $g$. The nearly
K\"ahler structure is {\em strict} if it is not K\"ahler.

For the purpose of this paper, a closed Einstein manifold $(M, g)$ is {\em linearly stable} if for all
transverse traceless (TT) symmetric $2$-tensors $h$, i.e., divergence-free and trace-free symmetric $2$-tensors,
the quadratic form
\begin{equation}  \label{stab-def}
{\mathscr Q}(h, h) = - \langle \nabla^* \nabla h - 2 \mathring{R}h, h \rangle_{L^2(M, g)} \leq 0.
\end{equation}
In the above $\mathring{R}$ is the action of the curvature tensor on symmetric $2$-tensors. $(M, g)$ is {\em linearly
unstable} if it is not linearly stable. The {\em coindex} of a quadratic form is the dimension of the
maximal subspace on which it is positive definite. More comments about stability will be given in section \ref{facts}.
Here we only mention that (a positive multiple of) the above quadratic form occurs in the second variation formula
of both the Einstein-Hilbert action and Perelman's $\nu$-entropy.

The main result of this article is

\begin{thm} \label{MainThm}
Let $(M, J, g)$ be a strict nearly K\"ahler $6$-manifold. If $b_2(M)$ or $b_3(M)$ is nonzero, then
$g$ is linearly unstable with respect to the Einstein-Hilbert action restricted to the space of
Riemannian metrics with constant scalar curvature and fixed volume. Hence it is also linearly
unstable with respect to the $\nu$-entropy of Perelman, and dynamically unstable with respect to
the Ricci flow.
\end{thm}

Note that an Einstein metric $g$ is {\em dynamically unstable} if there exists a non-trivial ancient
rescaled Ricci-flow $g_t,  -\infty < t \leq 0,$ such that $g_t$ converges modulo diffeomorphisms to $g$
in the pointed Cheeger-Gromov topology. The conclusion about dynamic instability in the above theorem
follows from Theorem 1.3 in \cite{Kr15}.

The proof of Theorem \ref{MainThm} actually shows that the coindex of the Einstein metric
$g$ (for either the Einstein-Hilbert action or $\nu$-functional) is $\geq b_2(M) + b_3(M)$.

By the theorem of Bonnet-Myers, a strict nearly K\"ahler $6$-manifold has finite fundamental group.
On the other hand, by pull-back any Riemannian cover of such a manifold also has a strict nearly
K\"ahler structure. From the properties of the transfer homomorphism, the corresponding Betti
numbers of any Riemannian cover are at least as large as those of the base. Hence the nearly K\"ahler metrics
on the covers are also linearly unstable.

At present there are very few examples of complete strict nearly K\"ahler $6$-manifolds. Recently,
Foscolo and Haskins produced the first non-homogeneous examples of such spaces \cite{FH17}.
One cohomogeneity one non-homogeneous nearly K\"ahler metric was produced on each of $S^6$
and $S^3 \times S^3$. Our result implies that the second metric is dynamically unstable.

In \cite{WW18} we showed that all the homogeneous nearly K\"ahler $6$-metrics other than the
isotropy irreducible space ${\rm G}_2/\SU(3) \approx S^6$ are linearly unstable.  Theorem \ref{MainThm}
provides some additional information for these cases. In the case of
$(\SU(2) \times \SU(2) \times \SU(2))/ \Delta \SU(2)$, it was shown in \cite{WW18} that the
first eigenspace of the nearly K\"ahler normal metric has dimension $12$ and the corresponding
eigenvalue is greater than $-2$ times the Einstein constant. Hence the normal
metric is linearly unstable
with respect to the $\nu$-entropy. However, the instability with respect to the Einstein-Hilbert
action was unresolved. Theorem \ref{MainThm} shows that this is also the case, and further that
the coindex of $g$ for the $\nu$-entropy is at least $12+2 = 14$. As for $\Sp(2)/(\Sp(1)\U(1)) = \C\PP^3$,
the Ziller metric was shown to be linearly unstable with respect to the Einstein-Hilbert action
by appealing to the properties of its canonical variation as a Riemannian submersion type metric.
The above theorem gives the instability without resorting to using fibrations or homogeneous geometry.
Finally, the coindex of the nearly K\"ahler normal metric on $\SU(3)/T^2$ is at least $2$ since
the second Betti number is $2$ in this case.

Finally we mention that Theorem \ref{MainThm} can be interpreted as a rigidity result in the form of

\begin{cor} \label{top-conseq}
Let $(M, J, g)$ be a simply connected, strict, nearly K\"ahler manifold that is linearly stable
with respect to the Einstein-Hilbert action. Then it is a rational homology sphere. In particular,
if $H_2(M, \Z)$ has no torsion, then $M$ is diffeomorphic to $S^6$.
\end{cor}

The corollary follows immediately from Theorem \ref{MainThm} by applying Wall's classification
of closed simply connected spin $6$-manifolds \cite{W66}. Recall that the absence of torsion
in the second integral homology implies that there is no torsion in integral homology, and
Wall showed that such manifolds are determined up to diffeomorphism by their integral
homology type and their first Pontryagin class.

After recalling in the next section the various notions of stability and those properties of
nearly K\"ahler manifolds that will be used in this paper, the proof of Theorem \ref{MainThm} will be given
in sections $3$ and $4$ for the respective cases of $b_2(M) \neq 0$ and $b_3(M) \neq 0$.

\medskip

\noindent{\bf Acknowledgements:} C. Wang gratefully acknowledges the support and wonderful working
condition of the Max Planck Institute for Mathematics in Bonn. M. Wang's research is partially supported
by NSERC Grant No. OPG0009421. Both authors like to thank X. Dai, S. Hall, F. He, J. Madnick, and G. Wei
for discussions and comments on an earlier version of the paper.

\section{\bf Preliminaries and Properties of Nearly K\"ahler Manifolds}  \label{facts}

We begin with explicit statements of conventions used in this paper because different authors use different
conventions for curvature quantities, and signs are of utmost importance for computations in the next sections.
We take the $(1, 3)$ curvature tensor to be $R_{X, Y} (Z) = [\nabla_X, \nabla_Y]Z - \nabla_{[X, Y]} \, Z$.
If $\{e_1, \cdots, e_n\}$ is an orthonormal frame, the $(0, 4)$-curvature tensor is taken to be
$R(e_i, e_j, e_k, e_l) = R_{ijkl}$. The sectional curvature determined by the $2$-plane $\{e_i, e_j\}$
is $R_{ijji}$. The action of the curvature on symmetric $2$-tensors is given by
$$(\mathring{R}h)_{ij} = - \sum_{p, q} R_{ipjq} h_{pq}.$$

Laplace-type operators will be consistent with the Laplace-Beltrami operator on functions
given as $\tr_g  ( {\rm Hess}_g)$, for which the eigenvalues are non-positive. When taking the norm
of $p$-forms, unless otherwise stated, we will use the tensor norm, in which one sums over all indices
without regard to order.

\subsection{Notions of linear stability of Einstein metrics}
We next describe in more detail the various notions of stability mentioned in the Introduction.
As is well-known, Einstein metrics on closed manifolds are critical points of the total scalar
curvature functional restricted to unit volume metrics. The second variation formula at an Einstein
metric consists of three parts. For directions tangent to the orbit of the diffeomorphism group,
the second variation is zero, and along directions corresponding to conformal changes, the second
variation is non-negative as a consequence of the theorem of Lichnerowicz-Obata. Therefore, it is
customary to associate linear stability of the Einstein-Hilbert functional with the second variation
restricted to the space of transverse traceless symmetric $2$-tensors (TT-tensors), which is the
tangent space of the space of unit volume constant scalar curvature metrics. By the work of Berger
and Koiso, on this space the second variation is given by $\frac{1}{2} {\mathscr Q}(h, h)$, where
$\mathscr Q$ is given by (\ref{stab-def}). Note that the operator $\nabla^* \nabla - 2 \mathring{R}$
on TT-tensors at an Einstein metric with Einstein constant $\Lambda$ is the same as $-(\Delta_L + 2 \Lambda \cdot \I)$
where $\Delta_L$ is the Lichnerowicz Laplacian and $\I$ is the identity operator. The notion of
linear instability given in the Introduction is equivalent to the condition
$\langle \nabla^* \nabla h - 2 \mathring{R}h, h \rangle_{L^2(M, g)} < 0$ for some nonzero TT-tensor $h$.

Einstein metrics with positive scalar curvature also occur among the critical points of Perelman's
$\nu$-entropy \cite{Pe02}. The second variation formula for this functional at an {\em Einstein} metric
was computed by H. D. Cao, R. Hamilton, and T. Illmanen \cite{CHI04} and explained in detail in \cite{CH15}.
(For the corresponding formula at a shrinking gradient Ricci soliton, see \cite{CZ12}.)
It likewise consists of three parts. Along directions orthogonal to the orbit of the diffeomorphism group
and along the space of TT-tensors, it agrees with that for the Einstein-Hilbert action
(up to some positive constant factor). Along directions
tangent to volume preserving conformal deformations, however, it can only have a positive definite subspace
provided there are eigenfunctions of the Laplace-Beltrami operator with eigenvalues larger than $-2\Lambda$.
In other words, unstable directions are given by these eigenfunctions and by TT-tensors which are
eigentensors of the Lichnerowicz Laplacian with eigenvalue $> -2\Lambda$.

Hence Einstein metrics (with positive scalar curvature) which are linearly unstable with respect to
the Einstein-Hilbert action are automatically linearly unstable with respect to the $\nu$-entropy.
As mentioned in the Introduction, Kr\"oncke's theorem implies that $\nu$-linearly unstable Einstein
metrics are dynamically unstable with respect to the Ricci flow.

\subsection{Properties of nearly K\"ahler $6$-manifolds}
For the convenience of the reader, we will summarise those properties of nearly K\"ahler
$6$-manifolds that will be used in the proof of Theorem \ref{MainThm}. For details and further
information, see \cite{Gr70}, \cite{Gru90}, \cite{BFGK91}, \cite{MNS08}, \cite{MS10}, \cite{V11}, and \cite{Fos17}.
We will assume that our nearly K\"ahler $6$-manifolds $(M, J, g)$ are complete, strict, and not isometric to
round $S^6$. We may normalize the Einstein metric $g$ to have Ricci curvature $n-1=5$. The first Chern
class of $J$ is zero, and so $M$ is spin.

In \cite{Gr70}, \cite{Gr72}, and \cite{Gr76}, Gray derived many identities involving the complex structure
$J$ and the curvature tensor $R$. Note that Gray's convention for curvature is opposite to ours. The following
subset of his identities will be used frequently in the next two sections:

\begin{equation} \label{curv-1}
R(X, Y, JZ, JW) = R(X, Y, Z, W) + g( (\nabla_X J)Y, (\nabla_Z J)W);
\end{equation}

\begin{eqnarray} \label{const-type}
    g( (\nabla_X J)Y, (\nabla_Z J)W) &=& g(X, Z)g(Y, W) - g(X, W)g(Y, Z) - \omega(X, Z) \omega(Y, W)\\
       &   &   \hspace{1cm}  + \,\,\,\omega(X, W) \omega(Y, Z)  \nonumber
\end{eqnarray}
where $\omega(X, Y) = g(JX, Y)$ is the fundamental $2$-form of the almost Hermitian structure;

\begin{equation}  \label{curv-2}
2g( (\nabla^2_{X, Y} J)Z, W) = -R(X, JY, Z, W) - R(X, JZ, W, X) - R(X, JW, Y, Z);
\end{equation}

\begin{equation} \label{J-2}
g( (\nabla^2_{X, X} J)Y, JZ) = - g((\nabla_X J)Y, (\nabla_X J)Z).
\end{equation}

Note that identity (\ref{const-type}), unlike the other three, is true in general only for nearly
K\"ahler $6$-manifolds (see Theorem 5.2 in \cite{Gr76}), and furthermore depends on the normalization
of the Ricci curvature of $g$ to be $5$. This normalization also fixes the constant $c$ in the Killing
spinor equation (\ref{KS-eqn}) to be $\frac{1}{2}$.

Another body of facts about nearly K\"ahler $6$-manifolds that we shall use result from viewing the
nearly K\"ahler structure as a special case of an $\SU(3)$ structure on $M$ (see \cite{Hi01}, \cite{CS04}).
Recall that the almost complex structure $J$ acts as an automorphism on the space of complex-valued
differential forms, and induces an orthogonal decomposition of this space into forms of type $(p, q)$. (Our convention
here is that of \cite{Bes87}, so that $J$ acts on a form of type $(p, q)$ as multiplication by $i^{q-p}$.)
A nearly K\"ahler structure is characterized by a pair $(\omega, \Omega)$ where $\omega$ is the fundamental $2$-form,
which is a real form of type $(1, 1)$, and $\Omega$ is a complex $3$-form of type $(3, 0)$. Let
$\Omega^{\pm}$ denote the real and imaginary parts of $\Omega$. Then $\omega$ and $\Omega^{+}$
are required to be stable in the sense that their ${\rm GL}(n, \R)$ orbits are open in the corresponding
spaces of real differential forms, and
$$ d \omega = 3 \Omega^{+}, \,\, d \Omega^{-} = -2 \omega \wedge \omega. $$
It follows that $\nabla \omega = \frac{1}{3} d \omega$. (Notice that if the nearly K\"ahler structure
were K\"ahler, then $\omega$ would be parallel.)

Regarding harmonic forms on $M$ we need the following result of Verbitsky:

\begin{thm} $($\cite{V11}, Theorem  6.2$)$  Let $(M, J, g)$ be a strict nearly K\"ahler $6$-manifold.
Then the space of harmonic $k$-forms is a direct sum of spaces ${\mathcal H}^{p, q}$ of harmonic forms of
type $(p, q)$ with $k=p+q$, and ${\mathcal H}^{p, q} = 0$ unless $p=q$ or $(p, q) = (2, 1)$
or $(1, 2)$. All harmonic $(1, 1)$-forms are primitive, as are all harmonic $3$-forms.
\end{thm}

An alternative proof of the above result can be found on p. 598 of \cite{Fos17}.

Associated to the $\SU(3)$ structure of a nearly K\"ahler $6$-manifold is the canonical hermitian
connection $\overline{\nabla}$ given by
$$  \overline{\nabla}_X Y = \nabla_X Y - \frac{1}{2} J(\nabla_X J)Y, $$
where $\nabla$ is the Levi-Civita connection of $g$.  Let $\overline{R}$ and $\overline{T}$ denote
respectively the curvature and torsion tensors of $\overline{\nabla}$. Then
$\overline{T}_X Y = -J(\nabla_X T)Y$, and $\overline{\nabla}\overline{T} = 0$. The curvature tensors
$\overline{R}$ and $R$ are related by  (see e.g. p. 253 of \cite{MS10})
\begin{equation} \label{curv-relation}
\begin{aligned}
   \overline{R}(X, Y, Z, W) =  & R(X, Y, Z, W) + \frac{1}{4} \left( g(X, Z) g(Y, W)-g(X, W) g(Y, Z) \right) +  \\
        &   \frac{1}{2}\, \omega(X, Y) \omega(Z, W) - \frac{3}{4} \left(\omega(X, Z) \omega(Y, W) - \omega(X, W) \omega(Y, Z) \right).
\end{aligned}
\end{equation}

We shall also need to refer to the decomposition of various tensor bundles into irreducible summands
with respect to the $\SU(3)$ structure. This material can be found for example in \cite{MNS08} or \cite{Fos17}.
We will identify spaces and their duals using the metric $g$. Because the connection $\overline{\nabla}$
is a connection on the principal bundle of the $\SU(3)$ structure, its curvature $\overline{R}$ acts trivially
on all trivial sub-bundles of these $\SU(3)$ decompositions.

Let $T$ denote the tangent bundle of $M$. Then $T \otimes \C = T_{(1, 0)} \oplus T_{(0, 1)}$. We have
$$ \Lambda^2 (T) =  \Lambda^2_6 \oplus (\I \oplus \Lambda^2_8)$$
where the subscripts represent as usual the real dimensions of the irreducible summands. The first
summand is the realification of $\Lambda^2 T_{(1, 0)}$  and consists of the skew $J$-invariant $2$-forms.
The trivial summand $\I$ is spanned by the fundamental $2$-form $\omega$. The third summand
consists of $J$-invariant $2$-forms which are primitive. In particular, all harmonic $2$-forms
are sections of this bundle, by Verbitsky's theorem.

Next, we have the orthogonal decomposition
$$ S^2(T) = S^2_{12} \oplus (\I \oplus S^2_8).$$
The bundle $S^2_{12}$ is the realification of $S^2(T_{(1, 0)})$ and consists of skew $J$-invariant
symmetric $2$-tensors. The other two irreducible summands consist of $J$-invariant symmetric $2$-tensors
with the metric $g$ generating the trivial summand. We emphasize here that $J$ is acting as an {\em automorphism}
on the tensors via $(J \cdot h)(X, Y) = h(J^{-1}X, J^{-1}Y) = h(JX, JY)$. $S^2_8$ and $\Lambda^2_8$ are
equivalent as $\SU(3)$ representations, and given a $2$-form $\eta$ the corresponding symmetric $2$-tensor
may be taken to be $h(X, Y) = \eta(JX, Y)$.

Finally, we need to consider the orthogonal decomposition
\begin{equation} \label{3form-decomposition}
\Lambda^3(T) =  \I \oplus \I \oplus (\Lambda^3_6 \oplus \Lambda^3_{12}).
\end{equation}
One may view the two trivial bundles as being spanned respectively by the forms $\Omega^{\pm}$, on
which the curvature $\overline{R}$ acts trivially. The remaining two summands consist of
realifications of forms of type $(2, 1)$.  Forms in $\Lambda^3_6$ consist of
exterior products $\alpha \wedge \omega$ where $\alpha$ is an arbitrary $1$-form.
The summand $\Lambda^3_{12}$ consists exactly of the primitive forms in
$\Lambda^3_6 \oplus \Lambda^3_{12}$.  Therefore, by Verbitsky's theorem, harmonic
$3$-forms on $M$ are sections of $\Lambda^3_{12}$.

It is further known that the elements in $\Lambda^3_{12}$ have the form $h^{\sharp} \cdot \Omega^{+}$
where $h^{\sharp}$ is a self-adjoint endomorphism of $TM$ which anticommutes with $J$ and
$\cdot$ denotes the action of an {\em endomorphism} on the $3$-form $\Omega^{+}$. For our purposes
it is more convenient to have an explicit expression of the inverse of this map. To derive this
association we note the following properties of $\Omega^{\pm}$:
\begin{equation} \label{Omega-prop}
\Omega^{\pm}(X, Y, Z) = - \Omega^{\pm}(X, JY, JZ); \,\,\,\,\, \Omega^{+}(JX, Y, Z) = -\Omega^{-}(X, Y, Z).
\end{equation}

\begin{prop}  \label{3forms-to-TT}
The $3$-forms in $\Lambda^3_6 \oplus \Lambda^3_{12}$ are characterized by the property
\begin{equation}  \label{3-form-invariance}
 \eta(X, Y, Z) = \eta(JX, JY, Z) + \eta(JX, Y, JZ) + \eta(X, JY, JZ).
\end{equation}
Furthermore, the maps
\begin{equation} \label{sigma-map}
 \sigma^{\pm}: \Lambda^3_6 \oplus \Lambda^3_{12} \longrightarrow S^2_{12}
\end{equation}
given by
\begin{equation}  \label{sigma-formula}
 \sigma^{\pm}(\eta)(X, Y) = \sum_{i, j} \, (\eta(X, e_i, e_j)\Omega^{\pm}(Y, e_i, e_j)+
        \eta(Y, e_i, e_j) \Omega^{\pm}(X, e_i, e_j))
\end{equation}
are surjective $\SU(3)$-equivariant maps with kernel $\Lambda^3_6$.  They satisfy the relation
$$\sigma^{\pm} ( h^{\sharp} \cdot \Omega^{\pm}) = - 8 h  $$
where $h^{\sharp}$ is the self-adjoint endomorphism corresponding to the symmetric $2$-tensor $h$.
\end{prop}

\begin{proof}
Let $\eta \in \Lambda^3(T)$ satisfy (\ref{3-form-invariance}). We claim it is orthogonal to $\Omega^{\pm}$. Indeed,
suppressing summation over indices $i, j, k$, we have
\begin{eqnarray*}
\eta_{ijk}\Omega^{\pm}_{ijk}  & = & \left( \eta(Je_i, Je_j, e_k) + \eta(e_i, Je_j, Je_k) + \eta(Je_i, e_j, Je_k) \right) \Omega^{\pm}_{ijk} \\
     & = & \eta(Je_i, Je_j, Je_k)\left(\Omega^{\pm}(e_i, e_j, Je_k) + \Omega^{\pm}(Je_i, e_j, e_k) + \Omega^{\pm}(e_i, Je_j, e_k)\right) \\
     & = & -\eta(Je_i, Je_j, Je_k)\left( \Omega^{\pm}(Je_i, Je_j, Je_k) + \Omega^{\pm}(Je_i, Je_j, Je_k)
           + \Omega^{\pm}(Je_i, Je_j, Je_k)  \right) \\
     & = & -3 \eta_{ijk} \Omega^{\pm}_{ijk},
\end{eqnarray*}
where we have used the $J$-invariance properties (\ref{Omega-prop}) of $\Omega^{\pm}$.

It is straightforward to check that (\ref{3-form-invariance}) holds for $\eta \in \Lambda^3_6$, using the
$J$-invariance of $\omega$.

Let $\eta = h^{\sharp} \cdot \Omega^{+} \in \Lambda^3_{12}$ where $h^{\sharp}$ is a self-adjoint
endomorphism that anticommutes with $J$. Using again the $J$-invariance properties of $\Omega^{\pm}$, one
easily obtains
$$ \eta(JX, JY, JZ) = - (h^{\sharp} \cdot \Omega^{-})(X, Y, Z). $$
Consider
\begin{eqnarray*}
\eta(JX, Y, Z) & = & -\Omega^{+}(h^{\sharp} JX, Y, Z) - \Omega^{+}(JX, h^{\sharp} Y, Z) - \Omega^{+}(JX, Y, h^{\sharp} Z) \\
        &=& -\Omega^{+}(J h^{\sharp}X, JY, JZ) + \Omega^{+}(JX, J h^{\sharp}Y, JZ) + \Omega^{+}(JX, JY, Jh^{\sharp}Z) \\
        &=& -\Omega^{-}(h^{\sharp}X, Y, Z) + \Omega^{-}(X, h^{\sharp}Y, Z) + \Omega^{-}(X, Y, h^{\sharp}Z),
\end{eqnarray*}
where we have used (\ref{Omega-prop}) and the fact that $h^{\sharp}$ anticommutes with $J$.
By cyclic permutation, it follows that
$$\eta(JX, Y, Z) + \eta(X, JY, Z) + \eta(X, Y, JZ) = -(h^{\sharp} \cdot \Omega^{-})(X, Y, Z) = \eta(JX, JY, JZ),$$
which implies (\ref{3-form-invariance}).

Moving to the maps $\sigma^{\pm}$, one easily checks that they are $\SU(3)$ equivariant because $\SU(3)$ fixes $\Omega^{\pm}$.
Since the range lies in $S^2 T$, the equivariance implies that it actually lies in $\Lambda^3_{12}$
and $\Lambda^3_6$ lies in the kernel. Restricted to the summand $\Lambda^3_{12}$, $\sigma^{\pm}$ is either $0$ or
multiplication by some nonzero constant. To check this, we choose a $J$-compatible orthonormal basis $\{e_k, 1 \leq k \leq 6 \}$
(i.e., $e_{2k} = J(e_{2k-1}$)) and consider the element $h=e^1 \otimes e^1 - e^2 \otimes e^2$.  We may take $\Omega^{+}$ to be the $3$-form
$$  {\rm Re}((e^1 + i e^2) \wedge (e^3 + ie^4) \wedge (e^5 +i e^6)) = e^{135} - e^{146} - e^{236} - e^{245}.$$
Then $\eta = h^{\sharp} \cdot \Omega^{+} = - (e^{135} - e^{146} + e^{236} + e^{245})$. It follows that
$$ \sigma^{+}(\eta)(e_1, e_1) = 4 ( \eta(e_1, e_3, e_5)\, \Omega^{+}(e_1, e_3, e_5)  +
\eta(e_1, e_4, e_6)\, \Omega^{+}(e_1, e_4, e_6)  ) = 4(-2) = -8. $$

An analogous argument gives the result for $\sigma^{-}$. This completes the proof of the Proposition.
\end{proof}


\section{\bf The $b_2(M) \neq 0$ Case}

In this section we will give a proof of the $b_2(M) \neq 0$ case of Theorem \ref{MainThm}. Recall that
Cao, Hamilton, and Illmanen observed in \cite{CHI04}, pp. 6-7, that a compact shrinking K\"ahler Ricci
soliton with $b_{1, 1} \geq 2$ is linearly unstable. Our result may be viewed as the analogue of this
observation for complete, strict, nearly K\"ahler $6$-manifolds. In this case, the fundamental $2$-form $\omega$
is not closed, and by Verbitsky's theorem, any harmonic $2$-form is pointwise orthogonal to $\omega$. Hence the
analogous condition is $b_2(M) > 0$ instead. Of course, since $\omega$ is not parallel, the corresponding
computations are more complicated.

\medskip

Let $\eta$ be a harmonic $2$-form and $h(X, Y):= \eta(JX, Y)$. By Verbitsky's theorem, $\eta$ is $J$-invariant
and primitive. So $h$ is a $J$-invariant symmetric $2$-tensor. Since $\eta$ is pointwise orthogonal to $\omega$, it follows
that $\tr_g h = 0$. Note also that $\|h \|^2 = \|\eta\|^2$ since we are using the tensor norm.

\begin{lemma}  $h$ is divergence-free.
\end{lemma}

\begin{proof} For each $x \in M$ we choose a local orthonormal frame $\{ e_i, 1 \leq i \leq 6 \}$
so that the Christoffel symbols vanish at $x$. Note that $\{ e_i^{\prime} := -Je_i \}$ is also
an orthonormal basis at $x$.

We first claim that
\begin{equation}  \label{first-claim}
 \sum_i \, (\nabla_{e_i} \eta)(Je_i, X) = \sum_i \, (\nabla_X  \eta)(Je_i, e_i).
\end{equation}
Indeed, by the nearly K\"ahler condition and the $J$-invariance of $\eta$,
\begin{eqnarray*}
\sum_i \, (\nabla_{e_i} \eta)(Je_i, e_j) &=&   \sum_i \, \nabla_{e_i} (\eta( Je_i, e_j)) - \eta((\nabla_{e_i}J)(e_i), e_j) \\
       &=& -\sum_i \, \nabla_{e_i}(\eta(e_i, Je_j))  \\
       &=&  (\delta \eta)(Je_j) - \sum_i \, \eta(e_i, (\nabla_{e_i} J)e_j)  \\
       &=&  0 - \sum_i \, \eta((\nabla_{e_j}J) e_i, e_i) \\
       &=&   - e_j(\tr_g h)) + \sum_i \, (\nabla_{e_j} \eta)(Je_i, e_i) \\
       &=&  \sum_i \, (\nabla_{e_j} \eta)(Je_i, e_i),
\end{eqnarray*}
where we have used the fact that $\eta$ is coclosed in the 4th equality and the fact that $h$ is trace-free
in the last equality.

On the other hand, from $d\eta (e_i, Je_i, e_j) = 0$, we obtain
\begin{equation*}
\sum_i \, (\nabla_{Je_i} \eta)(e_i, e_j) = \sum_i \, (\nabla_{e_i} \eta)(Je_i, e_j) + \sum_i \, (\nabla_{e_j} \eta)(e_i, Je_i) = 0,
\end{equation*}
by (\ref{first-claim}) above. But the left hand side equals
$$  - \sum_i \, (\nabla_{e_i^{\prime}} \eta)(J e_i^{\prime}, e_j) = (\delta h)(e_j), $$
by using the nearly K\"ahler condition once more. This proves the lemma.
\end{proof}

Next we analyse $\nabla^* \nabla h$. By straightforward computations and expressing $h$ in terms of $\eta$ we obtain
\begin{eqnarray*}
(\nabla^* \nabla h)_{ij} &=& - \sum_p \,  (\nabla_p \nabla_p h)_{ij}  \\
      &=& -\sum_p \, e_p(e_p(\eta(Je_i, e_j)) + \sum_p \, e_p(\eta(J(\nabla_p e_i), e_j)) + \sum_p e_p(\eta(Je_i, \nabla_p e_j)) \\
      &=& (\nabla^* \nabla \eta)(Je_i, e_j) -2 \sum_p \, (\nabla_p \eta)((\nabla_p J)e_i, e_j) - \sum_p \, \eta((\nabla_p \nabla_p J)(e_i), e_j).
\end{eqnarray*}
Since $\eta$ is harmonic, the usual Bochner formula for $2$-forms gives
$$ 0 = (\nabla^* \nabla \eta)_{ij} + 2 \sum_{p, q} \, R_{ipjq} \eta_{pq} + 2\Lambda \eta_{ij}. $$
Substituting this into the last expression for $(\nabla^* \nabla h)_{ij}$ and using the definition of $\mathring R(h)$, we get
\begin{eqnarray*}
(\nabla^* \nabla h)_{ij} &=& -2 \sum_{p, q} \, R(Je_i, Je_p, e_j, e_q)h_{pq} - 2 \Lambda h_{ij}
           -2 \sum_p \, (\nabla_p \eta)((\nabla_p J)e_i, e_j)  \\
           &   &  - \sum_p \, \eta((\nabla_p \nabla_p J)(e_i), e_j).  \\
\end{eqnarray*}
Note that at this point one immediately obtains the Cao-Hamilton-Illmanen instability result
for the Fano K\"ahler-Einstein case (with no dimension restrictions). Using further (\ref{curv-1})
and (\ref{const-type}), we obtain
\begin{eqnarray*}
 (\nabla^* \nabla h - 2 \mathring{R}h )_{ij} & =& -2 \Lambda h_{ij} -2 \sum_{p, q} \, g((\nabla_i J)e_p, (\nabla_j J)e_q)\, h_{pq} \\
          &  &  -2 \sum_p \, (\nabla_p \eta)((\nabla_p J)e_i, e_j) - \sum_p \, \eta((\nabla_p \nabla_p J)(e_i), e_j) \\
          &=& -10\, h_{ij} + 4\,h_{ij} -2 \sum_p \, (\nabla_p \eta)((\nabla_p J)e_i, e_j) - \sum_p \, \eta((\nabla_p \nabla_p J)(e_i), e_j),
\end{eqnarray*}
where we have also used the $J$-invariance of $h$ and the fact $\tr_g h = 0$.

We now use (\ref{J-2}) to evaluate the last term above. Then
\begin{eqnarray*}
- \sum_p \, \eta((\nabla_p \nabla_p J)(e_i), e_j) &=&  -\sum_{p, q} \, g((\nabla_p \nabla_p J)e_i, e_q) \eta_{qj} \\
    &=& -\sum_{p, q} \, g((\nabla^2_{e_p, e_p} J)e_i, Je_q) h_{qj}  \\
    &=& \sum_{p, q} \, g((\nabla_p J)e_i, (\nabla_p J)e_q) \, h_{qj} \\
    &=& \sum_{p, q} \, (\delta_{iq} - \delta_{pq} \delta_{ip} - \omega_{pp}\,\omega_{iq} + \omega_{pq}\, \omega_{ip})\, h_{qj} \\
    &=& 4\, h_{ij}
\end{eqnarray*}
where we used (\ref{const-type}) in the last equality above. Hence
\begin{eqnarray*}
 (\nabla^* \nabla h - 2 \mathring{R}h )_{ij} &=&  -\,2h_{ij}  -2 \sum_p \, (\nabla_p \eta)((\nabla_p J)e_i, e_j) \\
       &=& -2\,h_{ij} -2 \sum_{p, q} \, (\nabla_p \omega)(e_i, e_q) \, (\nabla_p \eta)(e_q, e_j).
\end{eqnarray*}

It remains to analyse the last term in the above. We have chosen not to substitute the $3$-form
$\Omega^{+}$ for  $ \nabla \omega$, in case parts of our computation can be used for other situations
where special $3$-forms are not available, e.g., Einstein hermitian manifolds. We have

\begin{eqnarray}
- \sum_{p, q} \, (\nabla_p \omega)(e_i, e_q) \, (\nabla_p \eta)(e_q, e_j) &=&
       - \sum_{p, q} \, (\nabla_p \omega)_{iq} (e_p(\eta(e_q, e_j))  \nonumber \\
       &=&  - \sum_{p, q} \, e_p \left( (\nabla_p \omega)_{iq}\, \eta(e_q, e_j)\right)
              + \sum_{p, q} \, (\nabla_p \nabla_p \omega)_{iq}\, \eta(e_q, e_j) \nonumber \\
       &=& - \sum_{p, q} \, e_p \left( (\nabla_p \omega)_{iq}\, \eta_{qj} \right)
               + \sum_{p, q} \, g((\nabla_p \nabla_p J)e_i, e_q) \, \eta_{qj}  \nonumber \\
       &=& - \sum_{p, q} \, e_p \left( (\nabla_p \omega)_{iq}\, \eta_{qj} \right) - 4 \, h_{ij} \label{2form-byparts}
        \end{eqnarray}
as before. It follows that
$$ \langle \nabla^* \nabla h - 2 \mathring{R}h, h \rangle = -10 \|h\|^2
         - 2 \sum_{i, j, p, q} \, e_p \left( (\nabla_p \omega)_{iq}\, \eta_{qj} \right) h_{ij}. $$

Now
\begin{eqnarray}
-  \sum_{i, j, p, q} \, e_p \left( (\nabla_p \omega)_{iq}\, \eta_{qj} \right) h_{ij} &=&
       - \sum_{i, j, p, q} e_p((\nabla_p \omega)_{iq}\, \eta_{qj}\, h_{ij}) - \sum_{i, j, p, q} (\nabla_p \omega)_{iq}\, \eta_{qj} \,(\nabla_p \eta)(e_i, Je_j)  \nonumber \\
     &  &   - \sum_{i, j, p, q} \, (\nabla_p \omega)_{iq}\, \eta_{qj}\, \eta(e_i, (\nabla_p J)e_j)   \label{bypartsII}
\end{eqnarray}
where the first term is a divergence term. For the second term, we have
\begin{eqnarray*}
- \sum_{i, j, p, q} (\nabla_p \omega)_{iq}\, \eta_{qj} \,(\nabla_p \eta)(e_i, Je_j) &=&
                - \sum_{i, j, p, q} (\nabla_p \omega)_{qi}\, \eta_{ij}\, (\nabla_p \eta)(e_q, Je_j) \\
              &=& - \sum_{i, j, p, q}  (\nabla_p \omega)_{iq} \, \eta(e_i, Je_j) (\nabla_p \eta)_{qj}  \\
              &=& \sum_{i, j, p, q}  (\nabla_p \omega)_{iq} \, h_{ij} (\nabla_p \eta)_{qj}.
\end{eqnarray*}
For the third term, we compute that
\begin{eqnarray*}
- \sum_{i, j, p, q} \, (\nabla_p \omega)_{iq}\, \eta_{qj}\, \eta(e_i, (\nabla_p J)e_j) &=&
      - \sum_{i, j, k, p, q} \, g((\nabla_p J)e_i, e_q) \, \eta_{qj}\, g( (\nabla_p J)e_j, e_k)\, \eta_{ik} \\
      & = & - \sum_{i, j, k, p, q} \, g((\nabla_i J)e_p, e_q) \, g( (\nabla_j J)e_p, e_k)\,  \eta_{qj} \,\eta_{ik} \\
      &=& - \sum_{i, j, k, p, q} \, g((\nabla_i J)e_q, e_p) \, g( (\nabla_j J)e_k, e_p)\,  \eta_{qj} \,\eta_{ik} \\
      &=&  - \sum_{i, j, k,  q} \, g((\nabla_i J)e_q, (\nabla_j J)e_k)\,  \eta_{qj} \,\eta_{ik} \\
      &=&  2 \|h \|^2
\end{eqnarray*}
where we used the nearly K\"ahler condition in the second equality, and (\ref{const-type}) and the pointwise
orthogonality between $\eta$ and $\omega$ to obtain the last equality.

Combining these calculations  with (\ref{2form-byparts}) and (\ref{bypartsII}), we obtain
\begin{equation*}
- \sum_{p, q} \, (\nabla_p \omega)_{iq} \, (\nabla_p \eta)_{qj}\, h_{ij}  =  \sum_{i, j, p, q}  (\nabla_p \omega)_{iq} \, h_{ij} (\nabla_p \eta)_{qj} - 2 \|h\|^2 + {\rm divergence \,\,\,term}.
\end{equation*}
Therefore,
$$- 2 \sum_{p, q} \, (\nabla_p \omega)_{iq} \, (\nabla_p \eta)_{qj}\, h_{ij} = -2 \|h \|^2 + {\rm divergence \,\,\,term},$$
which implies that
$$  \langle \nabla^* \nabla h - 2 \mathring{R}h, h \rangle_{L^2(M, g)} = -4 \|h\|^2_{L^2(M, g)}. $$
This shows that the quadratic form  $\mathcal Q$ is negative definite on the subspace of TT-tensors corresponding
to the harmonic $2$-forms on $M$.


\section{\bf The $b_3(M) \neq 0$ Case}
In this section we will prove the $b_3(M)\neq0$ case of Theorem \ref{MainThm} by constructing a destabilizing TT symmetric 2-tensor from any given harmonic 3-form via the isomorphism $\sigma^{+}: \Lambda^{3}_{12}\rightarrow S^{2}_{12}$ obtained in Proposition \ref{3forms-to-TT}.
Let $\eta\in \Lambda^{3}_{12}$ be a harmonic 3-form, and define $h_{\eta}\in S^{2}_{12}$ as
\begin{equation}\label{definition-of-h_eta}
h_{\eta}(X, Y):=\sigma^{+}(\eta)(X, Y)=\sum_{i,j}\left(\eta(X, e_{i}, e_{j})\Omega^{+}(Y, e_{i}, e_{j})+\eta(Y, e_{i}, e_{j})\Omega^{+}(X, e_{i}, e_{j})\right).
\end{equation}
We will show that $h_{\eta}$ is a destabilizing direction.

Recall $h_{\eta}\in S^{2}_{12}$ is skew $J$-invariant, i.e.
\begin{equation}  \label{skew-J-invarance}
-h_{\eta}(X, Y)=(J\cdot h_{\eta})(X, Y)=h_{\eta}(J^{-1}X, J^{-1}Y)=h_{\eta}(JX, JY).
\end{equation}
This implies $\tr_{g}(h_{\eta})=0$. Indeed,
\begin{equation*}
\tr_{g}(h_{\eta})=\sum_{i}h_{\eta}(e_{i}, e_{i})=-\sum_{i}h_{\eta}(Je_{i}, Je_{i})=-\tr_{g}(h_{\eta}),
\end{equation*}
since $\{Je_{1}, \cdots, Je_{6}\}$ is also a local orthonormal frame.

\begin{lemma}\label{TT-tensor-h_eta}
$h_{\eta}$ is divergence-free.
\end{lemma}
\begin{proof}
As before we still compute at a fixed but arbitrary point $x\in M$, with a local orthonormal frame $\{e_{1}, \cdots, e_{6}\}$ satisfying $\nabla_{e_{i}}e_{j}=0$ at $x$ for all $1\leq i, j\leq 6$. The negative divergence of $h_{\eta}$ is
\begin{eqnarray*}
-(\delta h_{\eta})(e_{j})
& = & \sum_{i}\,(\nabla_{e_{i}}h_{\eta})(e_{i}, e_{j}) \cr
& = & \sum_{i, p, q} \,e_{i}(\eta(e_{i}, e_{p}, e_{q})\Omega^{+}(e_{j}, e_{p}, e_{q})+\eta(e_{j}, e_{p}, e_{q})\Omega^{+}(e_{i}, e_{p}, e_{q}))\cr
& = & \sum_{i, p, q} \,\eta_{ipq}(\nabla_{e_{i}}\Omega^{+})_{jpq} + \sum_{i, p, q} \,(\nabla_{e_{i}}\eta)_{jpq}\Omega^{+}_{ipq}
       + \sum_{i, p, q} \,\eta_{jpq}(\nabla_{e_{i}}\Omega^{+})_{ipq},
\end{eqnarray*}
since $\delta\eta=0$.

Recall the identity (see, e.g. p. 64 in \cite{MNS08})
\begin{equation} \label{covariant-derivative-of-Omega+}
\nabla_{X}\Omega^{+}=-X^{\flat}\wedge\omega.
\end{equation}
Thus,
\begin{equation*}
(\nabla_{e_{i}}\Omega^{+})_{jpq}=-(e^{i}\wedge\omega)_{jpq}=-\delta^{i}_{j}\,\omega_{pq}+\delta^{i}_{p}\,\omega_{jq}-\delta^{i}_{q}\,\omega_{jp},
\end{equation*}
and
\begin{equation*}
\sum_{i}(\nabla_{e_{i}}\Omega^{+})_{ipq}=-\sum_{i}(e^{i}\wedge\omega)_{ipq}
=-\sum_{i}\,(\omega_{pq}-\delta^{i}_{p}\,\omega_{iq}+\delta^{i}_{q}\,\omega_{ip})=-4\omega_{pq}.
\end{equation*}

Moreover, for any fixed $1\leq j\leq 6$,
\begin{equation} \label{eta-omega-inner-product}
\sum_{p, q}\,\eta_{jpq}\,\omega_{pq}=2\langle i_{e_{j}}\eta, \omega \rangle = 2\langle \eta, e^{j}\wedge\omega \rangle =0,
\end{equation}
since $\eta\in\Lambda^{3}_{12}$, $e^{j}\wedge\omega\in\Lambda^{3}_{6}$, and the decomposition in (\ref{3form-decomposition}) is
pointwise orthogonal. Here $\langle \cdot , \cdot \rangle$ is the inner product of forms.

Consequently, the 1st term in the above expression of $-(\delta h_{\eta})(e_{j})$ vanishes, since
\begin{eqnarray*}
\sum_{i, p, q}\,\eta_{ipq}\,(\nabla_{e_{i}}\Omega^{+})_{jpq}
& = & \sum_{i, p, q}\,(-\eta_{ipq}\,\delta^{i}_{j}\,\omega_{pq}+\eta_{ipq}\,\delta^{i}_{p}\,\omega_{jq}-\eta_{ipq}\,\delta^{i}_{q}\,\omega_{jp})\cr
& = & -\sum_{p, q}\, \eta_{jpq}\,\omega_{pq}=0.
\end{eqnarray*}
Similarly, the 3rd term vanishes as well.

Finally, for the 2nd term, we have
\begin{eqnarray*}
\sum_{i, p, q}(\nabla_{e_{i}}\eta)_{jpq}\Omega^{+}_{ipq}
& = & \sum_{i, p, q}\Omega^{+}_{ipq}\left((\nabla_{e_{j}}\eta)_{ipq}-(\nabla_{e_{p}}\eta)_{ijq}+(\nabla_{e_{q}}\eta)_{ijp}\right)\cr
& = & \sum_{i, p, q}e_{j}\left(\Omega^{+}_{ipq}\eta_{ipq}\right)-\sum_{i, p, q}\eta_{ipq}(\nabla_{e_{j}}\Omega^{+})_{ipq}
-2\sum_{i, p, q}\Omega^{+}_{ipq}(\nabla_{e_{p}}\eta)_{ijq}\cr
& = & -\sum_{i, p, q}\eta_{ipq}(\nabla_{e_{j}}\Omega^{+})_{ipq} - 2\sum^{n}_{i, p, q=1}(\nabla_{e_{i}}\eta)_{jpq}\Omega^{+}_{ipq},
\end{eqnarray*}
where we used $d\eta=0$ in the first equality, and $\langle \Omega^{+}, \eta \rangle=0$ in the last equality.
Then combining with (\ref{covariant-derivative-of-Omega+}) and (\ref{eta-omega-inner-product}) again, we have
\begin{eqnarray*}
3\sum_{i, p, q}\,(\nabla_{e_{i}}\eta)_{jpq}\,\Omega^{+}_{ipq}
& = & -\sum_{i, p, q}\,\eta_{ipq}(\nabla_{e_{j}}\Omega^{+})_{ipq}\cr
& = & -\sum_{i, p, q}\,\eta_{ipq}\,(-\delta^{j}_{i}\,\omega_{pq}+\delta^{j}_{p}\,\omega_{iq}-\delta^{j}_{q}\,\omega_{ip})\cr
& = & 3\sum_{p, q}\, \eta_{jpq}\,\omega_{pq}=0.
\end{eqnarray*}

Thus $\delta (h_{\eta})=0$, and it proves the lemma.
\end{proof}

Now we claim:
\begin{equation}\label{h_eta_unstable}
(\nabla^{*}\nabla-2\mathring{R})h_{\eta}=-6h_{\eta}.
\end{equation}
This will complete the proof of the $b_{3}(M)\neq0$ case of Theorem \ref{MainThm}.

\medskip

{\bf Proof of (\ref{h_eta_unstable})}:
We still compute at a point $x\in M$ with a local orthonormal frame $\{e_{1}, \cdots, e_{6}\}$ satisfying $\nabla_{e_{i}}e_{j}= 0$ at $x$ for all $1\leq i,j\leq 6$.

By substituting in the definition of $h_{\eta}$ in (\ref{definition-of-h_eta}), straightforward calculations give
\begin{equation}\label{Laplace-of-h_eta1}
\begin{aligned}
(\nabla^{*}\nabla h_{\eta})_{jk}=
&-2\sum_{i, p, q}(\nabla_{e_{i}}\eta)_{jpq}(\nabla_{e_{i}}\Omega^{+})_{kpq} - 2\sum_{i, p, q}(\nabla_{e_{i}}\eta)_{kpq}(\nabla_{e_{i}}\Omega^{+})_{jpq}\\
& +\sum_{p, q}\left((\nabla^{*}\nabla\eta)_{jpq}\,\Omega^{+}_{kpq} + (\nabla^{*}\nabla\eta)_{kpq}\,\Omega^{+}_{jpq}\right)\\
& +\sum_{p, q}\left((\nabla^{*}\nabla\Omega^{+})_{jpq}\,\eta_{kpq} + (\nabla^{*}\nabla\Omega^{+})_{kpq}\,\eta_{jpq}\right)\\
& -2\sum_{i, p, q}\,\left(\Omega^{+}_{kpq}\,\eta(e_{j}, \nabla_{e_{i}}\nabla_{e_{i}}e_{p}, e_{q}) + \Omega^{+}_{jpq}\,\eta(e_{k}, \nabla_{e_{i}}\nabla_{e_{i}}e_{p}, e_{q})\right)\\
& -2\sum_{i, p, q}\,\left(\eta_{jpq}\,\Omega^{+}(e_{k}, \nabla_{e_{i}}\nabla_{e_{i}}e_{p}, e_{q}) + \eta_{kpq}\,\Omega^{+}(e_{j}, \nabla_{e_{i}}\nabla_{e_{i}}e_{p}, e_{q})\right).
\end{aligned}
\end{equation}

The last four terms are cancelled out, because
\begin{equation*}
g(\nabla_{e_{i}}\nabla_{e_{i}}e_{p}, e_{l}) = -g(e_{p}, \nabla_{e_{i}}\nabla_{e_{i}}e_{l})
\end{equation*}
implies
\begin{equation*}
\sum_{i, p, q} \,\eta_{jpq}\, \Omega^{+}(e_{k}, \nabla_{e_{i}}\nabla_{e_{i}}e_{p}, e_{q})
=-\sum_{i, p, q} \, \eta(e_{j}, \nabla_{e_{i}}\nabla_{e_{i}}e_{p}, e_{q})\, \Omega^{+}_{kpq},
\end{equation*}
and similarly for the other two terms.

For the sum of 1st and 2nd terms, we use the identities in (\ref{covariant-derivative-of-Omega+})
and (\ref{eta-omega-inner-product}), and then obtain
\begin{eqnarray*}
\sum_{i, p, q}(\nabla_{e_{i}}\eta)_{jpq}(\nabla_{e_{i}}\Omega^{+})_{kpq}
& = & \sum_{i, p, q}\,(\nabla_{e_{i}}\eta)_{jpq} \left( -\delta^{i}_{k}\,\omega_{pq} +
      \delta^{i}_{p}\,\omega_{kq} - \delta^{i}_{q}\,\omega_{kp} \right)\cr
& = & -\sum_{p, q} \,(\nabla_{e_{k}}\eta)_{jpq}\,\omega_{pq}\cr
& = & -\sum_{p, q} \left( e_{k}(\eta_{jpq}\,\omega_{pq}) - \eta_{jpq}(\nabla_{e_{k}}\,\omega)_{pq} \right)\cr
& = & \sum_{p, q} \,\eta_{jpq}\,\Omega^{+}_{kpq},
\end{eqnarray*}
since $\nabla\omega=\Omega^{+}$. We also used $\delta\eta=0$ in the 2nd equality. Then
\begin{eqnarray*}
-2\sum_{i, p, q}(\nabla_{e_{i}}\eta)_{jpq}(\nabla_{e_{i}}\Omega^{+})_{kpq} - 2\sum_{i, p, q}(\nabla_{e_{i}}\eta)_{kpq}(\nabla_{e_{i}}\Omega^{+})_{jpq}
& = & -2\sum_{p, q} \left( \eta_{jpq}\Omega^{+}_{kpq} + \eta_{kpq}\Omega^{+}_{jpq} \right) \\
&=& -2(h_{\eta})_{jk}.
\end{eqnarray*}

For the sum of 5th and 6th terms, the identity in (\ref{covariant-derivative-of-Omega+}) implies
\begin{eqnarray*}
\left(\nabla^{*}\nabla\Omega^{+}\right)_{jpq}
& = & -\sum_{i} (\nabla_{e_{i}}\nabla_{e_{i}}\Omega^{+})_{jpq}\cr
& = & -\sum_{i} e_{i}\left( (\nabla_{e_{i}}\Omega^{+})(e_{j}, e_{p}, e_{q}) \right)\cr
& = & (\nabla_{e_{j}}\omega)_{pq} - (\nabla_{e_{p}}\omega)_{jq} + (\nabla_{e_{q}}\omega)_{jp}\cr
& = & (d\omega)_{jpq} = 3\Omega^{+}_{jpq}.
\end{eqnarray*}
Then
\begin{equation*}
\sum_{p, q}\left( (\nabla^{*}\nabla\Omega^{+})_{jpq}\,\eta_{kpq} + (\nabla^{*}\nabla\Omega^{+})_{kpq}\,\eta_{jpq} \right) = 3(h_{\eta})_{jk}.
\end{equation*}

Thus, (\ref{Laplace-of-h_eta1}) becomes
\begin{equation}\label{Laplace-of-h_eta2}
(\nabla^{*}\nabla h_{\eta})_{jk}
 = (h_{\eta})_{jk}  + \sum_{p, q}\left( (\nabla^{*}\nabla\eta)_{jpq}\,\Omega^{+}_{kpq} + (\nabla^{*}\nabla\eta)_{kpq}\,\Omega^{+}_{jpq} \right).
\end{equation}

The Weitzenb\"ock formula
\begin{equation*}
\left( (d\delta+\delta d)\eta \right)_{jpq} = (\nabla^{*}\nabla\eta)_{jpq} + \sum_{i}\left( (R_{e_{i}e_{j}}\eta)_{ipq} - (R_{e_{i}e_{p}}\eta)_{ijq} + (R_{e_{i}e_{q}}\eta)_{ijp} \right)
\end{equation*}
together with $\Ric_{jl}=5g_{jl}$ imply
\begin{equation*}
\left( \nabla^{*}\nabla\eta \right)_{jpq}
= -15\eta_{jpq} - \sum_{i, l}R_{jpil}\,\eta_{ilq} - \sum_{i, l}R_{qpil}\,\eta_{ijl} - \sum_{i, l}R_{jqil}\,\eta_{ipl},
\end{equation*}
since $\eta$ is harmonic. Thus
\begin{eqnarray*}
& & \sum_{p, q}\,\left( (\nabla^{*}\nabla\eta)_{jpq}\,\Omega^{+}_{kpq} + (\nabla^{*}\nabla\eta)_{kpq}\,\Omega^{+}_{jpq} \right)\cr
& = & -15(h_{\eta})_{jk} - 2\sum_{p, q, i, l}\,R_{jpil}\,\eta_{ilq}\,\Omega^{+}_{kpq}
       - 2\sum_{p, q, i, l}\,R_{kpil}\,\eta_{ilq}\,\Omega^{+}_{jpq}\cr
& & +\sum_{p, q, i, l}\, R_{pqil}\,(\eta_{ijl}\,\Omega^{+}_{kpq} + \eta_{ikl}\,\Omega^{+}_{jpq}).
\end{eqnarray*}
Substituting this into (\ref{Laplace-of-h_eta2}) and using
\begin{eqnarray*}
( \mathring{R}h_{\eta} )_{jk}
 =  -\sum_{i, l}R_{jikl}(h_{\eta})_{il}
 =  -\sum_{p, q, i, l} R_{jikl}\,\left( \eta_{ipq}\,\Omega^{+}_{lpq} + \eta_{lpq}\,\Omega^{+}_{ipq} \right),
\end{eqnarray*}
we have
\begin{equation}\label{Lichnerowicz-Laplace-h_eta}
\begin{aligned}
( (\nabla^{*}\nabla-2\mathring{R})h_{\eta} )_{jk} =
& -14(h_{\eta})_{jk} + 2\sum_{p, q, i, l}\,R_{jikl}\,(\eta_{ipq}\,\Omega^{+}_{lpq} + \eta_{lpq}\,\Omega^{+}_{ipq})\\
& -2\sum_{p, q, i, l} \,\left(R_{jpil}\eta_{ilq}\,\Omega^{+}_{kpq} + R_{kpil}\,\eta_{ilq}\,\Omega^{+}_{jpq}\right)\\
& +\sum_{p, q, i, l} \,R_{pqil}\,(\eta_{ijl}\,\Omega^{+}_{kpq} + \eta_{ikl}\,\Omega^{+}_{jpq}).
\end{aligned}
\end{equation}

In the rest of the proof, we show
\begin{equation}\label{C}
\sum_{p, q, i, l} R_{pqil}(\eta_{ijl}\,\Omega^{+}_{kpq} + \eta_{ikl}\,\Omega^{+}_{jpq}) = 2(h_{\eta})_{jk},
\end{equation}
and
\begin{equation}\label{A+B}
2\sum_{p, q, i, l}\,R_{jikl}(\eta_{ipq}\,\Omega^{+}_{lpq} + \eta_{lpq}\,\Omega^{+}_{ipq})
-2\sum_{p, q, i, l}\, \left(R_{jpil}\,\eta_{ilq}\,\Omega^{+}_{kpq} + R_{kpil}\,\eta_{ilq}\,\Omega^{+}_{jpq}\right) = 6(h_{\eta})_{jk}.
\end{equation}
Then plugging (\ref{C}) and (\ref{A+B}) into (\ref{Lichnerowicz-Laplace-h_eta}) completes the proof of (\ref{h_eta_unstable}).

\medskip

{\bf Proof of (\ref{C})}.
By using identities in (\ref{curv-1}), (\ref{Omega-prop}) in the 2nd equality below, and (\ref{const-type}) in the 3rd equality below, we have
\begin{eqnarray*}
\sum_{p, q, i, l} R_{pqil}\,\eta_{ijl}\,\Omega^{+}_{kpq}
& = & \sum_{p, q, i, l} \,R(e_{i}, e_{l}, Je_{p}, Je_{q})\,\Omega^{+}(e_{k}, Je_{p}, Je_{q})\,\eta_{ijl}\cr
& = & -\sum_{p, q, i, l} \,\left( R_{ilpq} + g((\nabla_{e_{i}}J)e_{l}, (\nabla_{e_{p}}J)e_{q}) \right)\Omega^{+}_{kpq}\,\eta_{ijl}\cr
& = & -\sum_{p, q, i, l} \left( R_{ilpq} + \delta_{ip}\delta_{lq} - \delta_{iq}\delta_{lp} -\omega_{ip}\omega_{lq} + \omega_{iq}\omega_{lp} \right) \Omega^{+}_{kpq}\eta_{ijl}\cr
& = & -\sum_{p, q, i, l}\,R_{pqil}\,\Omega^{+}_{kpq}\,\eta_{ijl} - \sum_{p, q, i, l}( \delta_{ip}\delta_{lq} - \delta_{iq}\delta_{lp} -\omega_{ip}\omega_{lq} + \omega_{iq}\omega_{lp} )\Omega^{+}_{kpq}\,\eta_{ijl}.
\end{eqnarray*}
So
\begin{eqnarray*}
 2\sum_{p, q, i, l}\, R_{pqil}\,\eta_{ijl}\,\Omega^{+}_{kpq}
& = & - \sum_{p, q, i, l}\, ( \delta_{ip}\delta_{lq} - \delta_{iq}\delta_{lp} -\omega_{ip}\omega_{lq} + \omega_{iq}\omega_{lp})\,\Omega^{+}_{kpq}\,\eta_{ijl}\cr
& = & 2\sum_{p, q}\,\eta_{jpq}\,\Omega^{+}_{kpq} + \sum_{i, l}\,\Omega^{+}(e_{k}, Je_{i}, Je_{l})\,\eta_{ijl}
        - \sum_{i, l}\, \Omega^{+}(e_{k}, Je_{l}, Je_{i})\,\eta_{ijl}\cr
& = & 2\sum_{p, q}\,\eta_{jpq}\,\Omega^{+}_{kpq} - \sum_{i, l}\,\Omega^{+}_{kil}\,\eta_{ijl} + \sum_{i, l}\,\Omega^{+}_{kli}\,\eta_{ijl}\cr
& = & 4\sum_{p, q}\,\eta_{jpq}\,\Omega^{+}_{kpq}.
\end{eqnarray*}
Then switching indices $j, k$ gives
\begin{equation*}
\sum_{p, q, i, l}\,R_{pqil}\,\eta_{ikl}\,\Omega^{+}_{jpq} = 2\sum_{p, q}\,\eta_{kpq}\,\Omega^{+}_{jpq}.
\end{equation*}
Thus
\begin{equation*}
\sum_{p, q, i, l}\, R_{pqil}\,(\eta_{ijl}\,\Omega^{+}_{kpq} + \eta_{ikl}\,\Omega^{+}_{jpq}) = 2 \sum_{p, q}\, (\eta_{jpq}\,\Omega^{+}_{kpq} + \eta_{kpq}\,\Omega^{+}_{jpq}) = 2(h_{\eta})_{jk}.
\end{equation*}
This completes the proof of (\ref{C}).

\medskip

{\bf Proof of (\ref{A+B})}.
\begin{eqnarray*}
&   & 2\sum_{p, q, i, l}\,R_{jikl}\,(\eta_{ipq}\,\Omega^{+}_{lpq} + \eta_{lpq}\,\Omega^{+}_{ipq})
-2\sum^{n}_{p, q, i, l} \,\left(R_{jpil}\,\eta_{ilq}\,\Omega^{+}_{kpq} + R_{kpil}\,\eta_{ilq}\,\Omega^{+}_{jpq}\right)\cr
& = & 2\sum_{p, q, i, l}\, R_{jikl}\,\eta_{ipq}\,\Omega^{+}_{lpq} - 2\sum_{p, q, i, l}\, R_{jpil}\,\eta_{ilq}\,\Omega^{+}_{kpq}\cr
&   & + 2\sum_{p, q, i, l}\, R_{kijl}\,\eta_{ipq}\,\Omega^{+}_{lpq} - 2\sum_{p, q, i, l}\, R_{kpil}\,\eta_{ilq}\,\Omega^{+}_{jpq}\cr
& = & {\rm I} + {\rm II}.
\end{eqnarray*}
Here I and II denote respectively the sum of the first two terms and the sum of the last two terms.

We first proceed with the sum I.
\begin{eqnarray*}
{\rm I}
& = & 2\sum_{p, q, i, l}\, R_{jikl}\,\eta_{ipq}\,\Omega^{+}_{lpq} - 2\sum_{p, q, i, l}\, R_{jpil}\,\eta_{ilq}\,\Omega^{+}_{kpq}\cr
& = & 2\sum_{p, q, i, l} \,\left( R_{jikl}\,\eta_{ipq}\,\Omega^{+}_{lpq} + R_{ljip}\,\eta_{ipq}\,\Omega^{+}_{klq} \right)\cr
& = & 2\sum_{p, q, i} \,\eta_{ipq} \left( \Omega^{+}(R_{e_{j}e_{i}}(e_{k}), e_{p}, e_{q}) - \Omega^{+}(e_{k}, R_{e_{i}e_{p}}(e_{j}), e_{q}) \right)\cr
& = & 2\sum_{p, q, i} \eta_{ipq} \left( \Omega^{+}(R_{e_{j}e_{i}}(e_{k}), e_{p}, e_{q}) + \Omega^{+}(e_{k}, R_{e_{j}e_{i}}(e_{p}), e_{q}) + \Omega^{+}(e_{k}, R_{e_{p}e_{j}}(e_{i}), e_{q}) \right)\cr
& = & 2\sum_{p, q, i} \eta_{ipq} \left( -(R_{e_{j}e_{i}}\Omega^{+})_{kpq} - \Omega^{+}(e_{k}, e_{p}, R_{e_{j}e_{i}}(e_{q})) + \Omega^{+}(e_{k}, R_{e_{p}e_{j}}(e_{i}), e_{q}) \right)\cr
& = & -2\sum_{p, q, i}\, \eta_{ipq}\, (R_{e_{j}e_{i}}\Omega^{+})_{kpq} + 2\sum_{p, q, i}\,\eta_{ipq}\,\Omega^{+}(e_{k}, R_{e_{j}e_{i}}(e_{q}), e_{p}) + 2\sum_{p, q, i}\,\eta_{ipq}\,\Omega^{+}(e_{k}, R_{e_{p}e_{j}}(e_{i}), e_{q})\cr
& = & -2\sum_{p, q, i}\, \eta_{ipq} (R_{e_{j}e_{i}}\,\Omega^{+})_{kpq}.
\end{eqnarray*}
Here we used the Bianchi identity in 4th equality.

By applying the fact that the curvature $\overline{R}$ of the canonical Hermitian connection $\overline{\nabla}$ acts on $\Omega^{+}$ trivially, it becomes
\begin{eqnarray*}
{\rm I}
& = & -2\sum_{p, q, i} \eta_{ipq} ((R_{e_{j}e_{i}} - \overline{R}_{e_{j}e_{i}})\Omega^{+})_{kpq}\cr
& = & 2\sum_{p, q, i} \eta_{ipq} \big( \Omega^{+}((R_{e_{j}e_{i}}-\overline{R}_{e_{j}e_{i}})(e_{k}), e_{p}, e_{q}) + \Omega^{+}(e_{k}, (R_{e_{j}e_{i}} - \overline{R}_{e_{j}e_{i}})(e_{p}), e_{q})\cr
&   & + \Omega^{+}(e_{k}, e_{p}, (R_{e_{j}e_{i}} - \overline{R}_{e_{j}e_{i}})(e_{q})) \big)\cr
& = & 2\sum_{p, q, i} \eta_{ipq} \big( \Omega^{+}_{lpq}(R_{jikl} - \overline{R}_{jikl}) + \Omega^{+}_{klq}(R_{jipl} - \overline{R}_{jipl}) + \Omega^{+}_{kpl}(R_{jiql} - \overline{R}_{jiql}) \big)\cr
& = & \frac{1}{2}\sum_{p, q}\,\eta_{kpq}\,\Omega^{+}_{jpq} + \sum_{p, q}\,\eta_{jpq}\,\Omega^{+}_{kpq}
   + \frac{3}{2}\left( \sum_{p, q, i, l}\,\eta_{ipq}\,\Omega^{+}_{lpq}\,\omega_{il} \right)\omega_{jk}\cr
&   & - \sum_{p, q}\,\eta(Je_{j}, e_{p}, e_{q})\,\Omega^{+}(Je_{k}, e_{p}, e_{q})
      + \frac{3}{2}\sum_{p, q}\, \eta(Je_{k}, e_{p}, e_{q})\,\Omega^{+}(Je_{j}, e_{p}, e_{q})\cr
&   & -5\sum_{p, q}\,\eta(Je_{j}, e_{p}, e_{q})\,\Omega^{+}(e_{k}, Je_{p}, e_{q}).
\end{eqnarray*}
In the last equality, we used the identities in (\ref{curv-relation}) and (\ref{eta-omega-inner-product}), and $g(\eta, \Omega^{+})=0$.

Next, we deal with the last three terms in the above expression of I. The identities in (\ref{Omega-prop}) and (\ref{3-form-invariance}) imply
\begin{eqnarray*}
&   & \sum_{p, q}\, \eta(Je_{j}, e_{p}, e_{q})\,\Omega^{+}(Je_{k}, e_{p}, e_{q})\cr
& = & \sum_{p, q} \,\eta(Je_{j}, Je_{p}, e_{q})\,\Omega^{+}(Je_{k}, Je_{p}, e_{q})\cr
& = & -\sum_{p, q} \,\left( \eta_{jpq} - \eta(Je_{j}, e_{p}, Je_{q}) - \eta(e_{j}, Je_{p}, Je_{q}) \right)\,\Omega^{+}(e_{k}, e_{p}, e_{q})\cr
& = & -\sum_{p, q}\, \eta_{jpq}\,\Omega^{+}_{kpq} + \sum_{p, q}\, \eta(Je_{j}, e_{p}, Je_{q})\,\Omega^{+}(e_{k}, e_{p}, e_{q})
       + \sum_{p, q} \,\eta(e_{j}, Je_{p}, Je_{q})\,\Omega^{+}(e_{k}, e_{p}, e_{q})\cr
& = & -\sum_{p, q} \,\eta_{jpq}\,\Omega^{+}_{kpq} - \sum_{p, q}\, \eta(Je_{j}, e_{p}, Je_{q})\,\Omega^{+}(Je_{k}, e_{p}, Je_{q})
     - \sum_{p, q} \,\eta(e_{j}, Je_{p}, Je_{q})\,\Omega^{+}(e_{k}, Je_{p}, Je_{q})\cr
& = & -\sum_{p, q} \,\eta_{jpq}\,\Omega^{+}_{kpq} - \sum_{p, q} \,\eta(Je_{j}, e_{p}, e_{q})\,\Omega^{+}(Je_{k}, e_{p}, e_{q})
      - \sum_{p, q} \,\eta_{jpq}\,\Omega^{+}_{kpq}.
\end{eqnarray*}
So
\begin{equation*}
\sum_{p, q} \,\eta(Je_{j}, e_{p}, e_{q})\,\Omega^{+}(Je_{k}, e_{p}, e_{q}) = -\sum_{p, q} \,\eta_{jpq}\,\Omega^{+}_{kpq}.
\end{equation*}

Similar arguments show
\begin{equation*}
\sum_{p, q} \,\eta(Je_{k}, e_{p}, e_{q})\,\Omega^{+}(Je_{j}, e_{p}, e_{q}) = -\sum_{p, q} \,\eta_{kpq}\,\Omega^{+}_{jpq},
\end{equation*}
and
\begin{equation*}
\sum_{p, q}\, \eta(Je_{j}, e_{p}, e_{q})\,\Omega^{+}(e_{k}, Je_{p}, e_{q}) = -\sum_{p, q}\,\eta_{jpq}\,\Omega^{+}_{kpq}.
\end{equation*}

Thus
\begin{eqnarray*}
{\rm I}
& = & \frac{1}{2}\sum_{p, q}\,\eta_{kpq}\,\Omega^{+}_{jpq} + \sum_{p, q}\,\eta_{jpq}\,\Omega^{+}_{kpq}
   + \frac{3}{2}\left( \sum_{p, q, i, l}\,\eta_{ipq}\,\Omega^{+}_{lpq}\,\omega_{il} \right)\omega_{jk}\cr
&   & +\sum_{p, q}\, \eta_{jpq}\,\Omega^{+}_{kpq} - \frac{3}{2}\sum_{p, q} \,\eta_{kpq}\,\Omega^{+}_{jpq}
      + 5\sum_{p, q}\, \eta_{jpq}\,\Omega^{+}_{kpq}\cr
& = & -\sum_{p, q} \,\eta_{kpq}\,\Omega^{+}_{jpq} + 7\sum_{p, q}\,\eta_{jpq}\,\Omega^{+}_{kpq}
     + \frac{3}{2}\left( \sum_{p, q, i, l}\,\eta_{ipq}\,\Omega^{+}_{lpq}\,\omega_{il} \right)\omega_{jk}.
\end{eqnarray*}

Then by switching indices $j$ and $k$, we have
\begin{equation*}
{\rm II} = -\sum_{p, q}\,\eta_{jpq}\,\Omega^{+}_{kpq} + 7\sum_{p, q}\,\eta_{kpq}\,\Omega^{+}_{jpq}
     + \frac{3}{2}\left( \sum_{p, q, i, l}\,\eta_{ipq}\,\Omega^{+}_{lpq}\,\omega_{il} \right)\omega_{kj}.
\end{equation*}

Thus
\begin{equation*}
{\rm I} + {\rm II} = -(h_{\eta})_{jk} + 7 (h_{\eta})_{jk} = 6 (h_{\eta})_{jk}.
\end{equation*}
This completes the proof of (\ref{A+B}), as well as the proof of (\ref{h_eta_unstable}).
\qed

\end{document}